\documentclass[12pt]{amsart} 
\usepackage{euler}
\usepackage{euscript}
\usepackage{multicol} 
\usepackage{amssymb}

\usepackage{color}

\definecolor{amaranth}{rgb}{0.9, 0.17, 0.31}
\definecolor{ao}{rgb}{0.0, 0.0, 1.0}
\definecolor{ao(english)}{rgb}{0.0, 0.5, 0.0}
\definecolor{deepmagenta}{rgb}{0.8, 0.0, 0.8}

\usepackage{enumitem}

\usepackage[pdftex]{graphicx}
\usepackage{amsmath} 
\usepackage{times} 
 
\def\XXint#1#2#3{{\setbox0=\hbox{$#1{#2#3}{\int}$}
     \vcenter{\hbox{$#2#3$}}\kern-.5\wd0}}

\theoremstyle{plain}
\newtheorem{theorem}{Theorem}[section]
\newtheorem{proposition}{Proposition}[section]

\newtheorem{remark}{\bf Remark}[section]
\theoremstyle{definition}

\renewcommand{\epsilon}{\varepsilon}

\renewcommand{\phi}{\varphi}

\DeclareMathOperator\Arg{Arg}

\dedicatory{To I.M. Sigal}

\title
[Harmonic oscillator in the AB magnetic field]
{Time dependent Schr\"odinger equation for harmonic oscillator in the Aharonov-Bohm magnetic field}

\author{Jiyu Fan}
\address{Jiyu Fan:  Department of Mathematics, Imperial College London,
j.fan23@imperial.ac.uk}

\author{Ari Laptev}
\address{Ari Laptev:  Department of Mathematics, Imperial College London,
a.laptev@imperial.ac.uk}

\begin{document}

\maketitle

\begin{quote}
{\normalfont\fontsize{8}{10}
\selectfont{\bfseries  Abstract.}
We construct an approximation of the kernel of the solution of the time dependent Schr\"odinger equation whose Hamiltonian is a 2D harmonic oscillator in Aharonov-Bohm magnetic field. The main tools used here were established in  the paper of A. Laptev and I.M. Sigal \cite{LapSig} and also in \cite{LSV},  where the authors considered a class of Fourier Integral Operators with global complex phases approximating the fundamental solutions (propagators) for time-dependent Schr\"odinger equations. For the example considered in this paper we are able to find the main term in the approximation of the kernel that equals a version of the Mehler formula.
}
\end{quote}

\setcounter{equation}{0}
\section{Introduction}\label{Sec:1}

\noindent
Let us consider a harmonic oscillator in  Aharonov-Bohm (AB) magnetic field in $L^2(\Bbb R^2)$ 
\begin{equation}\label{H}
H _{\alpha,b,\omega} = \frac12 \left(-i \alpha\nabla - A\right)^2 + \frac{\omega^2}{2}|x|^2, \qquad x\in \Bbb R^2, \,\quad  \alpha, \omega>0,
\end{equation}
where 
$$
A(x) = \left(A_1(x) ,A_2(x)\right) = \left( -b\,\frac{x_2}{|x|^2}, b\, \frac{x_1}{|x|^2}\right).
$$
Here  the parameter $b\in\Bbb R$  is identified with the magnetic flux of the vector potential $A$ 
and $\alpha$ is a (small) semiclassical parameter. We also assume here that the operator $H _{\alpha,b,\omega}$ is understood as the Friedrichs extension defined by the closure from the class of functions $C_0^\infty(\Bbb R^2\setminus \{0\})$.

\medskip
\noindent
In this paper we propose a construction of a global Fourier Integral Operator (FIO)
giving an approximate fundamental solution (or propagator) to the equation
\begin{equation}\label{U}
i \alpha \frac{\partial }{\partial t} U_{\alpha,b,\omega}(t) = H_{\alpha,b,\omega} U_{\alpha,b,\omega}(t), \quad U_{\alpha,b,\omega}(0) = I,
\end{equation}
where $I$ is the identity operator in $L^2(\Bbb R^2)$.

\noindent
In the case $b\in \alpha \Bbb N$ the flux can be gauged away and in fact it is enough to consider the case where $0<b\le\alpha/2$ due to the symmetry $b\in[0,\alpha/2]$ and $b\in[\alpha/2, \alpha]$. 

\noindent
If for example $b=0$ and $\omega = 0$ then the kernel of the operator $U_{\alpha,0,0}$ can be easily constructed using the Fourier transform representing it as an oscillatory integral (kernel of a Fourier integral operator) 
$$
U_{\alpha,0,0}(t,x,y) = (2\pi \alpha)^{-2} \, \int_{\Bbb R^2} e^{i \varphi(x,y,\eta)/\alpha} \ d\eta,
$$
where $\varphi$ is the phase function
$$
\varphi(x,y,\eta)= (x-y)\cdot\eta - t \, |\eta|^2/2.
$$
In the theory of Fourier integral operators introduced in \cite{H1, H2, D}, 
the wave front set ${\rm WF}$ for oscillatory integrals coincides with the set in the cotangent bundle $T^*(\Bbb R^2\times \Bbb R^2)$, as follows 
$$
{\rm WF} (U_{\alpha,0,0}) =  \{ (x,\varphi_x; y,\varphi_y): \, \varphi_\eta = 0\} = \{ (x,\eta; y,-\eta): \, x= y+ t\eta\}.
$$
In the case $b=0$ and $\omega>0$, then the kernel $U_{\alpha,0,\omega}$ coincides with the well-known Mehler formula.

\medskip
\noindent
If $\omega= 0$ and $0<b\le \alpha /2$ the explicit formula for $U_{\alpha, b, 0}$ does not exist. In this case we deal with the propagator of the Aharonov-Bohm magnetic effect that states that the wave function of a charged quantum particle passing by a thin magnetic solenoid experiences a phase shift. This happens despite that there is no apparent interaction with the solenoid except through the
interaction of the particle with the ‘unphysical’ vector potential. This prediction
was made originally by Ehrenberg and Siday in 1949 \cite{ES}  and then again in
1959 by Aharonov–Bohm \cite{AB}. It cannot be explained in terms of classical mechanics, but was nevertheless experimentally verified (see \cite{BT}). It counts as one of the important quantum mechanical effects.

\medskip
\noindent
Denote by $h$ the hamiltonian related to the operator $H_\alpha$ 
\begin{equation}\label{h}
h(x,\xi) = \frac12 \left(\xi -A(x)\right)^2 + \frac{\omega^2}{2} |x|^2.
\end{equation}
In order to formulate the main result we consider the classical Hamilton-Jacobi system of equations 
 \begin{equation}\label{HJ}
\begin{cases}
\dot{x}^t = h_\xi (x^t,\xi^t) = \xi^t - A, & x^t |_{t=0} = y \not=0,\\
\dot{\xi}^t = - h_x(x^t,\xi^t) =  (\xi^t - A) \cdot (A)_{x} - \omega^2 \,x^t, & \xi^t |_{t=0} = \eta,
\end{cases}
\end{equation}
where we also assume that 
 \begin{equation}\label{y_eta}
y\wedge\eta = y_1\eta_2 - y_2\eta_1 \not = b.
\end{equation}
Let us introduce the action function
\begin{equation}\label{action}
S(t,y,\eta) = \int_0^t \left(h_\xi(x^s,\xi^s) \cdot \xi^s - h(x^s,\xi^s)\right)\, ds.
\end{equation} 
Let $\mathcal B>0$ and let 
\begin{equation}\label{phi}
\varphi(t,x,y,\eta) = S(t,y,\eta) + (x-x^t)\cdot \xi^t + i \mathcal B \,|x-x^t|^2/2.
\end{equation}
We also denote by $Z$ the matrix 
\begin{equation}\label{Z}
Z= Z(t,y,\eta) = \xi^t_\eta - i \mathcal B x^t_\eta.
\end{equation} 

\medskip
\noindent 
We introduce cut-off functions $\theta\in C^\infty(\Bbb R^2)$, such that 
\begin{equation*}
\theta(x) = 
\begin{cases}
1, & |x| >2,\\
0, & |x| < 1,
\end{cases}
 \end{equation*}
and $\theta_\varepsilon (x) = \theta(x/\varepsilon)$.

\medskip
\noindent
Now we are ready to formulate our main result:

\begin{theorem}\label{main1}
For any integer $N\ge 0$ and $\varepsilon>0$
there is an operator $U_N (t)$
such that the Schwartz kernel of $U_N (t)$ is of the form 
\begin{equation}\label{U_NN}
U_N (t,x,y) = (2\pi \alpha)^{-2} \int_{\Bbb R^2} e^{i\varphi / \alpha}  u_N \,d\eta,
\end{equation}
where $u_N = u_N (t,y,\eta , \alpha) = \sum\limits_{k=0}^N \alpha^k u_{(k)}(t,y,\eta) $
with 
$$ 
u_{(0)}(t,y,\eta) =  \theta_\varepsilon(x^t(y,\eta))\theta_\varepsilon(y) \, 
(\det Z(t,y,\eta))^{1/2}.
$$
Here we choose the continuous in $t$ branch of the square root of the determinant ${\rm det}\, Z$ 
assuming that ${\rm arg} \det Z|_{t=0} = 0$.

\noindent
The operator $U_N (t)$ approximates the propagator $U_{\alpha,b,\omega}(t)$ in the following sense 
$$ 
\left( i \alpha \frac{\partial }{\partial t}  - H_{\alpha,b,\omega} \right) U_N(t) = O(\alpha^{N-1}).
$$
\end{theorem}

\bigskip
\noindent
Originally the theory of FIO was introduced in the papers of  \cite{H2,DH}, see also \cite{H1,D}, where the authors considered real-valued phase functions. In order to avoid the problem with focal points of Hamiltonian flow, they used products of local FIOs. Another version of this approach can be found in V.P. Maslov \cite{Mas}. FIO theory with complex-valued phase functions allows one to construct the global FIO theory avoiding the problem of focal points, see \cite{LSV, SV} and this approach was extended to non-homogeneous symbols in \cite{LapSig}. 
Oscillatory integrals with complex phase functions in a different form are used in \cite{BU, CR, MS, R} and many others.

\medskip
\noindent
The study of properties of solutions of time-dependent Schr\"odinger equations is a classical problem. There are only a few examples where such solutions are known in the exact form. One of them is a well known Mehler formula for the case of harmonic oscillator  \cite{M}, see also  \cite{CFKS}. Some generalisations of such formulae were also obtained in the papers \cite{LapSig}, where the authors studied Laplacians with quadratic potentials, Laplacians with constant magnetic fields, including harmonic oscillators in a constant magnetic field, see also the case of the Heisenberg Laplacian considered in \cite{BGG}.

\medskip
\noindent
In Section \ref{Sec:5} using the stationary phase method we are able to integrate the main term in \eqref{U_NN} and obtain
\begin{theorem}
Let $\sin(\omega t)\not= 0$. Then 
\begin{equation}\label{main_term_UN} 
U_N(t,x,y)=\frac{\omega e^{iS(t,y,\eta_0)/\alpha}}{2\pi i\alpha\sin(\omega t)}(1+O(\alpha)),
\end{equation}
where 
$$
\eta_0(t,x,y)=A(y)+\frac{\omega}{\sin(\omega t)}(x-\cos(\omega t)y)
$$ 
and 
\begin{multline*}
 S(t,y,\eta_0(t,x,y))= \frac{\omega}{2\sin(\omega t)}\left(\cos(\omega t)(|x|^2+|y|^2)-2x\cdot y\right)\\
 +b(\Arg x-\Arg y+2\ell\pi).
\end{multline*}
\end{theorem}

\medskip
\noindent
The main term in \eqref{main_term_UN} almost coincides with the Mehler formula. The only difference is  the additional term $b(\Arg x-\Arg y+2\ell\pi)$ in the phase that is the influence of the AB magnetic field. Here $\ell$ is determined by a continuous lift.

\smallskip
\noindent
Finally, we would like to note in \cite{Han} the author used polar coordinates and, by applying a separation of variables, could find the eigenvalues and eigenfunctions of the operator \eqref{H} explicitly.


\section{Solutions of the Hamilton-Jacobi equations}\label{Sec:2}

\noindent
In this section we consider some properties of solutions of the Hamilton-Jacobi system of equations \eqref{HJ}.
Differentiating the first equation in \eqref{HJ} and using the second one we find
\begin{multline}\label{secdd}
\ddot{x}_k^t = \dot{\xi}_k^t - \nabla{A_k}\cdot{\dot{x}}^t =
(\xi^t - A) \cdot (A)_{x_k}  -\omega^2 x_k^t
- \nabla{A_k}\cdot{\dot{x}}^t 
\\= 
\sum_{j=1}^2 (\xi_j - A_j) \left[ (A_j)_{x_k} - (A_k)_{x_j}\right] - \omega^2 x_k^t
 \\= (-1)^{k+1} 2\pi b (\xi_{3-k} - A_{3-k}(x^t)) \, \delta(x^t) - \omega^2 x_k^t, \quad k=1,2, 
\end{multline}
with the initial conditions 
$$
x^t |_{t=0} = y, \qquad \dot{x}^t |_{t=0} = \eta - A(y).
$$
If $x^t$ does not meet zero then the solutions of the Hamiltonian equations can be found explicitly 
\begin{align}\label{HJsol}
\begin{split}
&x^t = \cos(\omega t) \,y + \frac{\sin(\omega t)}{\omega} (\eta -A(y)), \\
&\xi^t =  \dot{x}^t + A(x^t) = -\omega \sin(\omega t) \,y +  \cos(\omega t) \left(\eta - A(y)\right) + A(x^t). 
\end{split}
\end{align}

\begin{remark}
From the equations \eqref{HJsol} we immediately obtain that $(x^t, \xi^t)$ are periodic trajectories whose periods are $2\pi /\omega$.
\end{remark} 

\medskip
\noindent
Let us consider the equation $x^t(y,\eta) = 0$, namely,
\begin{align} \label{xt0}
\begin{split}
& \cos(\omega t) \,y_1 + \frac{\sin(\omega t)}{\omega}  \, \left(\eta_1 +b\,  \frac{y_2}{|y|^2}\right) = 0, \\
 & \cos(\omega t) \, y_2 + \frac{\sin(\omega t)}{\omega}  \, \left(\eta_2 - b\,  \frac{y_1}{|y|^2}\right) = 0. 
 \end{split}
\end{align} 
Multiplying the first equation by $y_2$, the second one by $y_1$ and subtracting them from each other we 
have
\begin{equation}\label{eq1}
\sin(\omega t) \left( y \wedge \eta - b\right) = 0,
\end{equation} 
where $y \wedge \eta= y_1\eta_2 - y_2 \eta_1$.

\noindent
Besides, multiplying the first equation in \eqref{xt0}  by $y_1$, the second one by $y_2$ and adding them , we obtain
\begin{equation}\label{eq22}
\cos(\omega t) \, |y|^2 + \frac{\sin(\omega t)}{\omega} \,\eta\cdot y = 0.
\end{equation} 
Note that if in \eqref{eq1} $\sin(\omega t) = 0$ then the equation \eqref{eq22} cannot be fulfilled. This implies that the equation $y \wedge \eta = b$ plays a prime role.

\medskip
\noindent
We can now interpret the identity \eqref{eq1} .
Let us introduce the angular momentum of the flow
$$
L  = x^t\wedge \dot{x}^t.
$$

\begin{proposition}\label{AngMom}
The angular momentum $L$ is independent of $t$. 
Moreover, the equation $x^t(y,\eta) = 0$  implies $L= L(y,\eta)  = 0$.
\end{proposition}
\begin{proof}
Clearly if $x^t(y,\eta) \not=0$ then using \eqref{secdd} 
$$
\frac{\partial}{\partial t} L = x^t \wedge \ddot{x}^t  = -\omega^2 \, x^t\wedge x^t = 0
$$
and thus $L$ is constant in $t$.

\medskip
\noindent
Using \eqref{eq1} $x^t(y,\eta) = 0$ we find $ y \wedge \eta = b$ and due to 
the Hamilton-Jacobi equation \eqref{HJ} we obtain
$$
L = x^t\wedge \dot{x}^t = x^t\wedge \xi^t - x^t\wedge A(x^t).
$$
Since 
$$
x^t\wedge A(x^t) = \frac{b}{|x^t|^2} \left[ (-x_2^t)^2 + (x_1^t)^2\right] = b
$$
we conclude
$$
L= L\big|_{t=0} = x^t \wedge \xi^t\big|_{t= 0} - b = 0.
$$ 
\end{proof}

\begin{remark}\label{inv}
\noindent
\smallskip
Assuming $L=0$ we have  $y\wedge \eta = b$. Then the equation \eqref{eq22} 
\begin{equation*}
\cos(\omega t) \, |y|^2 + \frac{\sin(\omega t)}{\omega} \,\eta\cdot y = 0.
\end{equation*} 
gives us time $t= t(y,\eta)$ when $x^t(y,\eta) = 0$.
\end{remark}

\medskip
\noindent
The conservation of the angular momentum tells us that if
$L(0)\not= 0$ (or $ y \wedge \eta \not= b$), we must have $L(t) = L(0)\not= 0$ for all
$t\ge0$. Consequently $x^t$  stays away from the origin for all $t\ge0$. On the other
hand, if $y\wedge\eta = b$, it means the instantaneous velocity of $x$ can point exactly at
$0$ either inward or outward. Therefore, $x$ moves
along a purely radial (i.e. the force has no tangential component) direction that
reaches the origin in finite time, see the graph below.

%

\begin{figure}[!htbp]
\centerline{\includegraphics[scale=.57]{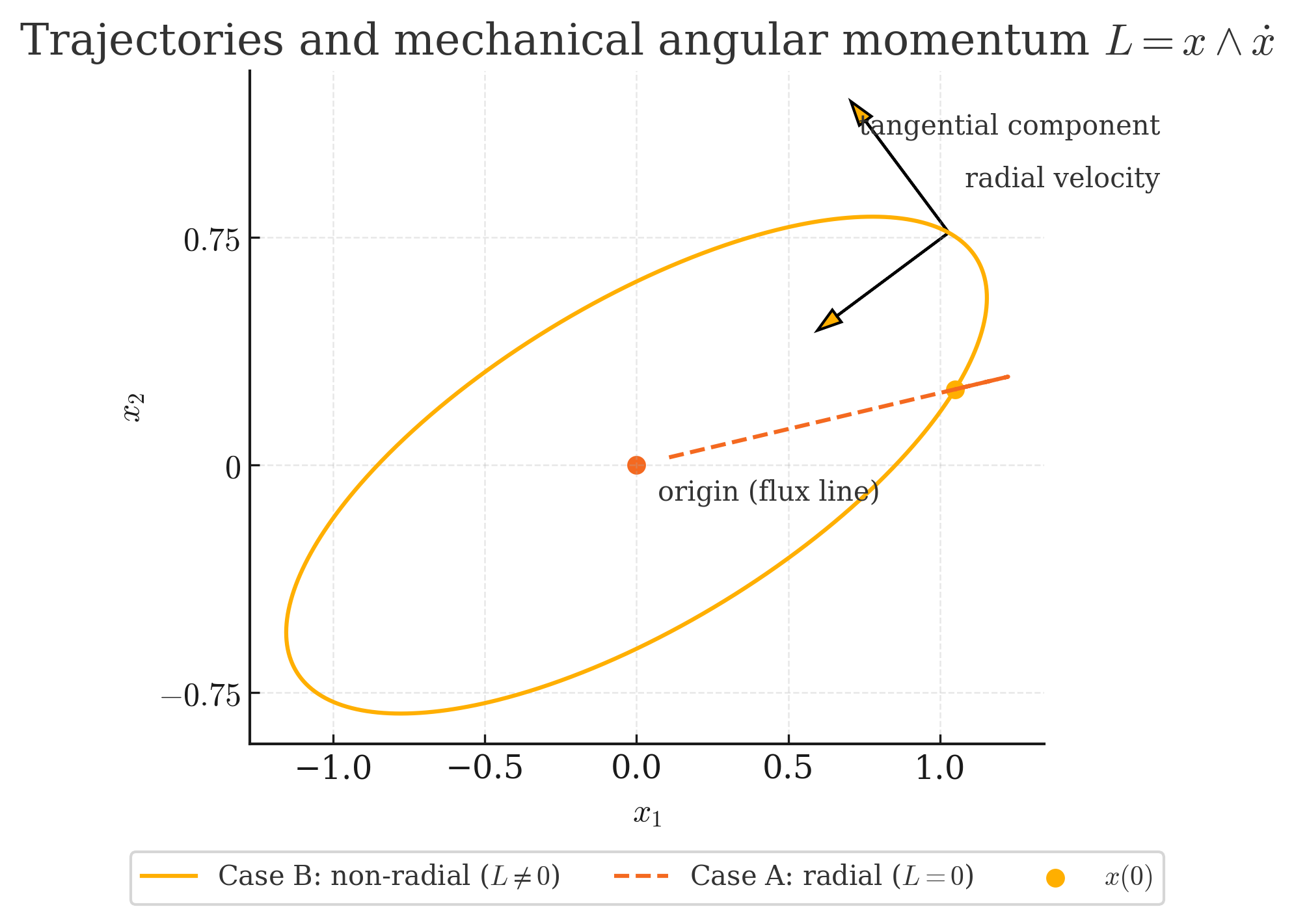}}
\end{figure}

\begin{remark}
The equation for $x^t$ obtained in \eqref{HJsol}
$$
x^t(y,\eta) = \cos(\omega t) \,y + \frac{\sin(\omega t)}{\omega} (\eta -A(y))
$$
defines an ellipse centered at the origin that is non-degenerate if and only if $L\neq 0$.
\end{remark}


\section{Action and phase functions}\label{Sec:3}

\noindent
Using the definition of the action function $S$ defined in \eqref{action} for the Hamiltonian $h$ we find 
\begin{align*}
S(t,y,\eta) &= \int_0^t \left(h_\xi(x^s,\xi^s) \cdot \xi^s - h(x^s,\xi^s)\right)\, ds\\
&=
\int_0^t \left( (\xi^s - A(x^s))\cdot  \xi^s  - \frac12 (\xi^s - A(x^s))^2 - \frac12 \omega^2\,|x^s|^2 \right)\, ds
\\
&=
\frac12\, \int_0^t \left( |\xi^s|^2 - |A(x^s)|^2  -  \omega^2\,|x^s|^2 \right)\, ds.
\end{align*}
 Since $\xi = (\xi-A) + A$ we have
 $$
 |\xi^s|^2 = |\xi^s - A(x^s)|^2 +2 (\xi^s - A(x^s) )\cdot A(x^s) + |A(x^s)|^2
 $$
 and therefore 
$$
|\xi^s|^2 - |A(x^s)|^2 = |\xi^s - A(x^s)|^2  + 2\left( \xi^s -A(x^s)\right) \cdot A(x^s).
$$
Moreover, using $x^s= y \cos(\omega s) + \frac{\sin(\omega s)}{\omega} (\eta - A(y))$ we find
\begin{multline}\label{|x|}
|x^s|^2 =  |y|^2 \, \cos^2(\omega s) +  \frac{\sin^2(\omega s)}{\omega^2} 
\left(|\eta|^2- 2b(y\wedge\eta - b/2) |y|^{-2}\right) \\
+ \frac{\sin(2\omega s)}{\omega} y\cdot \eta.
\end{multline}
Similarly, as $\xi^s-A(x^s)=\dot{x}^s=-\omega\sin(\omega s)y+\cos(\omega s)(\eta-A(y))$, we have
 \begin{multline*}
 |\xi^s-A(x^s)|^2=\omega^2\sin^2(\omega s)|y|^2+\cos^2(\omega s)(|\eta|^2-2b(y\wedge\eta-b/2)|y|^{-2})\\
 -\omega\sin(2\omega s)y\cdot \eta
\end{multline*}
and 
$$
2 (\xi^s - A(x^s)) \cdot A(x^s) = 2(-\omega\sin(\omega s)y+\cos(\omega s)(\eta-A(y)))\cdot A(x^s).
$$
We have $y\cdot A(y) = 0$ and also 
\begin{align*}
 -\sin(\omega s)\,  y &\cdot A(x^s)  \\
 & = 
 - \sin(\omega s) \,y_1  b \left(y_2 \cos(\omega s) + \left(\eta_2 - b \frac{y_1}{|y|^2}\right) \frac{\sin(\omega s)}{\omega}
 \right) |x^s|^{-2}\\
&  
- \sin(\omega s) \, y_2   b\left(-y_1\cos(\omega s) -  \left(\eta_1 + b\frac{y_2}{|y|^2}\right) \frac{\sin(\omega s)}{\omega} \,\right) |x^s|^{-2}\\
& = b \, \omega^{-1} (y\wedge \eta - b)\, |x^s|^{-2} \sin^2(\omega s).
\end{align*}
Similarly 
\begin{equation*}
 \cos(\omega s) \,(\eta - A(y)) \cdot A(x^s)  
= b ( y\wedge \eta - b)\, |x^s|^{-2} \cos^2(\omega s).
\end{equation*}
Finally we conclude 

\begin{align*}
S(t,y,\eta)=&\frac{\cos(2\omega t)-1}{2} y\cdot\eta \\
&+\frac{\sin(2\omega t)}{4\omega}(|\eta-A(y)|^2-\omega^2|y|^2)+b(y\wedge \eta-b)\int_0^t|x^s|^{-2}ds.
\end{align*}
According to Proposition \ref{AngMom} the angular momentum $L = x^t\wedge \dot{x}^t$ is a constant in $t$ and equals $L(y,\eta) = y\wedge \eta - b$. Therefore 
\begin{multline*}
(y\wedge\eta -b) \int_0^t |x^s|^{-2} \, ds =  \int_0^t  (x^s\wedge \dot{x}^s)  |x^s|^{-2} \, ds \\
= 
\int_0^t \frac{x_1^s \dot{x}_2^s - x_2^s \dot{x}_1^s}{(x_1^s)^2 + (x_2^s)^2}\, ds = 
{\rm Arg}\,(x^t)-{\rm Arg}\,(y)+2 \ell\pi. 
\end{multline*}

\medskip

\begin{proposition}\label{SProp}
The action function  \eqref{action} satisfies the following identity 
\begin{equation}\label{Seta}
S_\eta = x^t_\eta \xi^t \quad {\rm and} \quad S_y = x_y^t \xi^t - \eta.
\end{equation}
\end{proposition} 
\begin{proof}
Both sides of each identity in \eqref{Seta} are equal to zero as $t=0$. Then using \eqref{HJ}
\begin{align*}
\frac{\partial}{\partial t} S_\eta &= \frac{\partial}{\partial \eta} (h_\xi \cdot \xi - h) \\
&= \dot{x}^t_\eta \xi^t +   \dot{x}^t \xi_\eta^t  - x_\eta^t h_x - h_\xi \xi_\eta^t \\
&= \dot{x}^t_\eta \xi^t + x^t_\eta \dot{\xi}^t = \frac{\partial}{\partial t} \left( x_\eta^t \xi^t \right).
\end{align*} 
This implies the first identity in \eqref{Seta} since $x^t_\eta \xi^t|_{t=0} = 0$.

\noindent
Moreover, analogously we have
\begin{align*}
\frac{\partial}{\partial t} S_y &= \frac{\partial}{\partial y} (h_\xi \cdot \xi - h) \\
&= \dot{x}^t_y \xi^t +   \dot{x}^t \xi_y^t  - x_y^t h_x - h_\xi \xi_y \\
&= \dot{x}^t_y \xi^t + x^t_y \dot{\xi}^t = \frac{\partial}{\partial t} \left( x_y^t \xi^t \right).
\end{align*} 
and using the initial $x^t_y \xi^t|_{t=0} = \eta$ we obtain the second identity in \eqref{Seta}.

\end{proof} 

\medskip
\begin{proposition}\label{A_x1}
The matrix $h_{x,\xi} = - \partial_x A(x) = - A_x$ is self-adjoint and 
$$
\det A_x = - b^2 \,|x|^{-4} \quad {\rm and} \quad {\rm Tr}  \,A_x = 0.
$$
The eigenvalues of $A_x$ are equal to
$$
\lambda_1(A_x) = b\, |x|^{-2} \quad {\rm and} \quad \lambda_2(A_x) = - b\, |x|^{-2}.  
$$
\end{proposition}
\begin{proof}
The statement follows immediately by differentiating the vector field
$$
A(x) = b\, \left(-\frac{x_2}{|x|^2}, \frac{x_1}{|x|^2}\right).
$$  
Namely,
\begin{equation}\label{A_x}
A_x(x)= 
b\, \begin{pmatrix}
\frac{2x_1x_2}{|x|^4} & \frac{x_2^2 -x_1^2}{|x|^4}\\
\\
\frac{x_2^2 -x_1^2}{|x|^4} & -\frac{2x_1x_2}{|x|^4}
\end{pmatrix} = (A_x(x))^*,
\end{equation}
\end{proof}

\medskip
\noindent
Let $\varphi$ be the phase function defined in \eqref{phi}
\begin{equation*}
\varphi(t,x,y,\eta) = S(t,y,\eta) + (x-x^t)\cdot \xi^t + i \mathcal B\,|x-x^t|^2/2 . 
\end{equation*}

\medskip
\noindent
We now consider the matrix defined in \eqref{Z}
\begin{equation*}
Z(t,y,\eta) = \varphi_{x\eta}(t,x,y,\eta)|_{x=x^t(y,\eta)} 
= \xi^t_\eta  - i \mathcal B x^t_\eta.
\end{equation*}

\medskip
\begin{proposition}\label{Zinv}
If $\mathcal B>0$ then the matrix $Z$ defined in \eqref{Z} is a non-degenerate matrix-function for all $t\ge0$. Moreover 
\begin{equation}\label{detZ}
\det Z = ( \cos(\omega t) - i \mathcal B \omega^{-1} \sin(\omega t))^2 - b^2 \omega^{-2} |x^t|^{-4} \sin^2(\omega t)
\end{equation}
and the eigenvalues $\mu_{1,2}$ of $Z$ are equal to 
\begin{equation}\label{lambda12}
\mu_{1,2} =  \cos(\omega t) - i \mathcal B \omega^{-1} \sin(\omega t)   \pm b \omega^{-1} |x^t|^{-2} \sin(\omega t).
\end{equation}
\end{proposition}

\begin{proof}
From the explicit form of the solutions $x^t$ and $\xi^t$ \eqref{HJsol} we find
\begin{align*}
&x^t_\eta = \omega^{-1} \sin(\omega t) \, \Bbb I \quad {\rm and} \\ 
& \xi^t_\eta = \cos(\omega t) \,\Bbb I +A_x (x^t) \, \omega^{-1}\sin(\omega t),
\end{align*}
where $\Bbb I$ is the $2\times2$ identity matrix. Using \eqref{A_x} and simple computations we obtain \eqref{detZ}. 
Besides the trace of matrix $Z$ equals 
$$
{\rm Tr} \, Z = 2\cos(\omega t)  - 2i \mathcal B \, \omega^{-1}\sin(\omega t).
$$
This together with \eqref{detZ} allows us to conclude that the eigenvalues $\mu_{1,2}$ of the matrix $Z$ satisfy \eqref{lambda12}.

\noindent
In order to prove that $Z$ is not degenerate it is suﬃcient to show that the eigenvalues of $Z  Z^*$ are never equal to zero.

\noindent
Note first  that 
\begin{equation}\label{x_xi}
\xi^t_\eta (x^t_\eta)^* =  x^t_\eta (\xi^t_\eta)^*.
\end{equation}
Indeed, $x^t_\eta|_{t=0} = 0$ and thus the above identity is fulfilled for $t=0$. Besides,
we have 
\begin{align*}
& \frac{\partial}{\partial t} x^t_\eta = x^t_\eta h_{x\xi} + \xi^t_\eta h_{\xi\xi},
& \frac{\partial}{\partial t} \xi^t_\eta = -x^t_\eta h_{xx} - \xi^t_\eta h_{\xi x}.
\end{align*}
Using these equations we find
$$
 \frac{\partial}{\partial t}\left( \xi^t_\eta (x^t_\eta)^* \right) = 
  \frac{\partial}{\partial t}\left( x^t_\eta (\xi^t_\eta)^*\right)
 $$
 and consequently \eqref{x_xi} is true for all $t\ge0$. 
 
\noindent
Moreover due to \eqref{x_xi} we obtain
\begin{equation*}
Z  Z^* = (\xi^t_\eta - i \mathcal B x^t_\eta)  ((\xi^t_\eta)^* + i \mathcal B (x^t_\eta)^* ) 
= \xi^t_\eta (\xi^t_\eta)^* + \mathcal B^2 x^t_\eta (x^t_\eta)^* .
\end{equation*}
Finally, we have 
if $\omega t\not=\pi k$, $k\in \Bbb Z$, then $x^t_\eta (x^t_\eta)^*=  \omega^{-2}\sin^2(\omega t) \not= 0$  and 
if $\omega t =\pi k$, then $\xi^t_\eta = \cos(\omega t) \Bbb I =  \pm \Bbb I$ and 
$\xi^t_\eta (\xi^t_\eta)^* =  \Bbb I$.
\end{proof} 

\medskip
\begin{remark}
Note that if we wish to choose $\varphi$ to be a real phase function $(\mathcal B=0)$, then we have a problem when the matrix $Z$ degenerates at time $t$ at which one of the eigenvalues $\mu_{1,2}$ equals zero.  
\end{remark}

\medskip
\begin{proposition}\label{phi0}
Solutions of the equation $\varphi_\eta = 0$ are subsets of the set 
$$
\left\{(x,y,\eta): \, x\wedge y = - \frac{\sin(\omega t)}{\omega} \,(y\wedge \eta - b)\right\}.
$$
\end{proposition}
\begin{proof}
Using Proposition \ref{SProp} we have 
\begin{equation*}
\varphi_\eta = S_\eta - x^t_\eta \xi^t + (x-x^t)\cdot (\xi^t_\eta  - i \mathcal B\,x_\eta^t) = (x-x^t) \cdot (\xi^t_\eta  - i \mathcal B\,x_\eta^t).
\end{equation*}
Since the matrix $\xi^t_\eta  - i \mathcal B\,x_\eta^t $ is not degenerate the equation $\varphi_\eta = 0$ implies   
\begin{equation}\label{eq12}
x- \cos(\omega t) \,y - \frac{\sin(\omega t)}{\omega}  \,(\eta - A(y))= 0 .
\end{equation}
The equation in \eqref{eq12} is equivalent to the system 
\begin{align*}
 & x_1 - \cos(\omega t) y_1 -  \frac{\sin(\omega t)}{\omega}  \left( \eta_1 + b \frac{y_2}{|y|^2}\right) = 0,\\
 & x_2 - \cos(\omega t) y_2 - \frac{\sin(\omega t)}{\omega}  \left( \eta_2 - b \frac{y_1}{|y|^2}\right) = 0
\end{align*} 
and therefore
$$
x\wedge y - \frac{\sin(\omega t)}{\omega}  (-y\wedge\eta + b) =0.
$$
\end{proof} 


\section{Fourier Integral Operator}\label{Sec:4}

\subsection{Substitution}

\noindent
Let us consider the wave equation
\begin{equation}\label{Schr} 
\left( i\, \alpha\, \frac{d}{dt} - H_{\alpha,b,\omega}\right) U_N(t) = R_N(t), 
\end{equation} 
where $R_N$ is an $\alpha$-FIO with the phase $\varphi$ and a symbol 
$\alpha^{N+2} r_N$.

\noindent
For the respective kernel of the integral operator $U_N$ we use the same notation $U_N= U_N(t,x,y)$ that we shall find in the  form of the following oscillatory integral
\begin{equation*}
U_N(t,x,y) = (2\pi\alpha)^{-2} \int_{\Bbb R^2} e^{i \varphi(t,x,y,\eta)/\alpha} u_N(t,y,\eta)\, d\eta,
\end{equation*} 
where the phase function $\varphi$ is defined in \eqref{phi}.

\noindent
Substituting $U_N$ into \eqref{Schr} we obtain
\begin{multline*} 
\left(i \alpha \, \frac{\partial }{\partial t} - H_{\alpha,b,\omega}  \right) U_N(t,x,y) \\
= 
\left( i \alpha \, \frac{\partial }{\partial t} -\frac12 (-i\alpha \nabla_x - A)^2  - \frac{\omega^2}{2} |x|^2 \right) U_N(t,x,y) \\
= (2\pi)^{-2} \int_{\Bbb R^2} e^{i\varphi(t,x,y,\eta)/\alpha} \left(
i\alpha\frac{\partial}{\partial t} u_N- 
(g + O(\alpha^2)) \,u_N \right) \, d\eta,
\end{multline*} 
where 
$$
g= g (t,x,y,\eta) = \varphi_t + h(x,\varphi_x) + \alpha \mathcal B.
$$
Applying \eqref{HJ} for the derivative of the phase function $\varphi_t $ we have
\begin{multline}\label{phi_t}
\varphi_t = h_\xi(x^t, \xi^t) \cdot \xi^t - h(x^t,\xi^t) 
 - \dot{x}^t\cdot \xi^t + (x-x^t)\cdot \dot{\xi}^t
-i \dot{x}^t \cdot (x-x^t) \mathcal B\\
 = - h(x^t,\xi^t) - (x-x^t) h_x(x^t,\xi^t) 
 -i \mathcal B \,h_\xi(x^t,\xi^t) \cdot (x-x^t).
\end{multline} 
Using the Taylor expansion we find
\begin{multline}\label{h(x,phi_x)}
h(x,\varphi_x) = h(x, \xi^t + i \mathcal B\,(x-x^t) ) \\
= h(x^t,\xi^t) + h_x(x^t, \xi^t) (x-x^t) 
+ i \mathcal B \,h_\xi(x^t, \xi^t) \cdot (x-x^t)  \\
+ 
\frac12 \sum_{|\nu_1 + \nu_2| = 2} \left(\partial_x^{\nu_1}\partial_\xi^{\nu_2} h\right)
(x-x^t)^{\nu_1} (i\mathcal B \,(x-x^t))^{\nu_2} + O(|x-x^t|^3).
\end{multline} 
Summing up the expressions \eqref{phi_t} and \eqref{h(x,phi_x)}, we find that the terms with $\alpha^0$ and with $(x - x^t)$ cancel. The rest can be rewritten as
$$
g = g_1 + \alpha \mathcal B + O(|x-x^t|^3),
$$
where
$$
g_1 = \frac12 \sum_{|\nu_1 + \nu_2| = 2} \left(\partial_x^{\nu_1}\partial_\xi^{\nu_2} h\right)
(x-x^t)^{\nu_1} (i\mathcal B\, (x-x^t))^{\nu_2}. 
$$

\smallskip
\subsection{Integration by parts} 
Note that 
$$
\frac{\partial}{\partial\eta} e^{i\varphi/\alpha} = \frac{i}{\alpha}  e^{i\varphi/\alpha} \varphi_\eta = 
\frac{i}{\alpha} e^{i\varphi/\alpha} \left( S_\eta - x^t_\eta \xi^t + \xi_\eta(x-x^t) -i \mathcal B\,x^t_\eta (x-x^t)\right).
$$
Due to Proposition \ref{SProp} $S_\eta - x^t_\eta \xi^t = 0$. Moreover \eqref{Z} implies that $Z =  Z(t,y,\eta)$ and we have 
$$
\frac{\partial}{\partial \eta} e^{i\varphi/\alpha} = \frac{i}{\alpha}  e^{i\varphi/\alpha} (x-x^t) Z(t,y,\eta).
$$
Proposition \ref{Zinv} implies that the matrix $Z$ is invertible and thus
$$
(x-x^t) e^{i\varphi/\alpha} = -i\alpha Z^{-1} \frac{\partial}{\partial \eta} e^{i\varphi/\alpha}. 
$$
Applying this equation and integrating twice by parts in order to convert all powers of 
$(x-x^t)$ into powers of $\alpha$, we obtain
\begin{multline*}
\int_{\Bbb R^2} e^{i\varphi/\alpha} g_1 u_N\, d\eta  = \frac{\alpha}{2i} \int_{\Bbb R^2} 
e^{i\varphi/\alpha} \Big \{ {\rm Tr}\, \left[ Z^{-1} (h_{xx} + i 2 \mathcal B\, h_{x\xi} - \mathcal B^2 ) \,x^t_\eta \right] u_N  \\
+  \sum_{|\beta| \le 2} O(\alpha) \partial^\beta u_N \Big\}.
\end{multline*} 
Using the Hamilton-Jacobi equations \eqref{HJ} we find
$$
\dot{x}^t_\eta = x^t_\eta h_{x\xi} + \xi^t_\eta h_{\xi\xi}, \qquad 
\dot{\xi}^t_\eta = -x^t_\eta h_{xx} - \xi^t_\eta h_{\xi x}.
$$
This implies 
\begin{multline*}
x^t_\eta (h_{xx} + i 2 \mathcal B\, h_{x\xi} - \mathcal B^2 )  = 
(x^t_\eta h_{xx} + \xi^t_\eta h_{x\xi}) + i \mathcal B\, (x^t_\eta h_{x\xi} + \xi^t_\eta ) \\
- (\xi^t_\eta - i \mathcal B\, x^t_\eta) h_{x\xi} - (\xi^t_\eta - i \mathcal B\, x^t_\eta) i \mathcal B 
= -\dot{Z} - Z h_{x\xi} - i\mathcal B\, Z.
\end{multline*}
Due to  \eqref{A_x} we have  ${\rm Tr} \,h_{x\xi} = 0$ and finally obtain
\begin{multline*} 
\left(i \alpha \, \frac{\partial }{\partial t} - H_{\alpha,b,\omega}  \right) U_N(t,x,y)\\
= 
i(2\pi)^{-2} \, \alpha^{-1} \, \int_{\Bbb R^2} e^{i\varphi/\alpha} 
\left[ \frac{\partial}{\partial t} - \frac12 {\rm Tr} \left(Z^{-1} \dot{Z}  \right)
+ O(\alpha)\right] \,u_N\, d\eta.
\end{multline*}
Therefore the equation \eqref{Schr} is reduced to  
\begin{equation}\label{symb1}
i\alpha \left[ \frac{\partial}{\partial t} - \frac12 {\rm Tr} \left(Z^{-1} \dot{Z}  \right)
+ O(\alpha)\right] \, u_N = \alpha^{N+2} r_N.
\end{equation}
Assuming $u_N = \sum_{k=0}^N \alpha^k u_{(k)}$ and equating in \eqref{symb1} the terms containing the same power of the parameter $\alpha$, we obtain the recurrent system of transport equations
\begin{equation}\label{symb2}
\frac{\partial}{\partial t} u_{(k)}  - \frac12 {\rm Tr} \left(Z^{-1} \dot{Z}  \right)
\,  u_{(k)} =O(1) u_{(k-1)}.
\end{equation}
Here by $O(1)$ we denote $u(-1) = 0$ and $O(1)$ is a differential operator in $\eta$ with coefficients of order 1 with $O(1) = \alpha^{-1} O(\alpha)$ and $O(\alpha)$ is the operator appearing in \eqref{symb1}. 

\noindent
In order to solve the transport equations \eqref{symb2} we need to define the initial conditions when solving \eqref{symb2}. For that we consider 
$$
U_N(0,x,y) = (2\pi \alpha)^{-2} \int_{\Bbb R^2} e^{i(x-y)\eta/\alpha} 
e^{-\mathcal{B}|x-x^t|^2/2\alpha} u_N(0,y,\eta) \, d\eta.
$$ 
The initial conditions for the functions $u_{(k)}$ are obtained by expanding the term 
$\exp(-\mathcal B|x-x^t|^2/2\alpha)$ in the Taylor series and integrating by parts replacing the terms with $x-x^t$ with the respective terms involving powers of $\alpha$.

\noindent
In particular, the first term in the mentioned Taylor expansion equals $1$, and we immediately obtain that 
$$
u_{(0)}(t,y,\eta) =  \theta_\varepsilon(y) \theta_\varepsilon(x^t(y,\eta)) \, 
(\det Z(t,y,\eta))^{1/2}.
$$
This finally leads us to the expression 
\begin{multline}\label{U_N}
U_N(t,x,y) = (2\pi\alpha)^{-2} \int_{\Bbb R^2} e^{i\varphi(t,x,y,\eta)/\alpha} 
 \theta_\varepsilon(y) \theta_\varepsilon(x^t(y,\eta))\\
\times \left((\det Z(t,y,\eta))^{1/2} + \sum_{k=1}^N \alpha^k u_{(k)}(t,y,\eta)\right) \, d\eta.
\end{multline}
%

\medskip
\section{A Mehler type formula} \label{Sec:5}

\noindent
Assume $\sin(\omega t)\neq 0$, we work locally where the trajectory avoids the singularity. Fix $(t,x,y)$, the stationary point in $\eta$ is given by the equation
$$
\varphi_\eta=(x-x^t)Z,
$$
which is equivalent to
$$x=x^t(y,\eta).$$
As $x^t_\eta=\frac{\sin(\omega t)}{\omega}I$, the unique solution is 
\begin{align}\label{critical_points}
    \eta_*(t,x,y)=A(y)+\frac{\omega}{\sin(\omega t)}(x-\cos(\omega t)y).
\end{align}

\begin{theorem}
    Let $K$ be a compact subset of the region where $\sin(\omega t)\neq 0$, the critical point \eqref{critical_points} is valid and the corresponding trajectory avoids the origin. Then, uniformly on $K$, we have
    \begin{align}\label{new}
        U_N(t,x,y)=\frac{\omega e^{iS(t,y,\eta_0)/\alpha}}{2\pi i\alpha\sin(\omega t)}\sum_{k=0}^N\alpha^kC_k(t,x,y)+O(\alpha^{N}),
    \end{align}
    where the explicit expression of $S(t,y,\eta_0(t,x,y))$ is given by
    \begin{multline}\label{S}
       S(t,y,\eta_*(t,x,y))= \frac{\omega}{2\sin(\omega t)}\left(\cos(\omega t)(|x|^2+|y|^2)-2x\cdot y\right)\\
       +b(\Arg x-\Arg y+2\ell\pi),
    \end{multline}
and the coefficients $C_j$ are the ordinary complex stationary phase Taylor coefficients in the variable $\eta$, explicitly, for $0\leq m\leq N$
    \begin{align*}
        C_m(t,x,y)=(\det Z(t,y,\eta_*))^{-1/2}\sum_{k=0}^mL_{m-k}u_{(k)}(t,y,\eta)|_{\eta=\eta_*(t,x,y)}.
    \end{align*}
    Here
    \begin{align}\label{L}
        L_jf=\sum_{\mu=0}^{2j}\frac{i^{-j}2^{-j-\mu}}{(j+\mu)!\mu!}\left[\left\langle \varphi^{-1}_{\eta\eta}D_\eta,D_\eta\right\rangle^{j+\mu}(\varphi^\mu_{\geq 3} f)\right]_{\eta=\eta_*},\quad f\in C^\infty(\mathbb{R}^2),
    \end{align}
    and
\begin{multline*}
\varphi_{\geq 3}(t,x,y,\eta)=\varphi(t,x,y,\eta)-\varphi(t,x,y,\eta_*)\\
-\frac{1}{2}(\eta-\eta_*)\cdot\varphi_{\eta\eta}(t,x,y,\eta_*)(\eta-\eta_*).
\end{multline*}
In particular, in the region where the cutoffs are identically one, one has
    \begin{align*}
        U_N(t,x,y)=\frac{\omega e^{iS(t,y,\eta_*)/\alpha}}{2\pi i\alpha\sin(\omega t)}(1+O(\alpha)).
    \end{align*}
\end{theorem}

\begin{proof}
    We apply Taylor's formula to the $\eta$ integral in the original formula. First write
    $$\eta=\eta_*+q,$$
    then
    $$\varphi(t,x,y,\eta_*+q)=\varphi(t,x,y,\eta_*)+\frac{1}{2}q\cdot\varphi_{\eta\eta}(t,x,y,\eta_*)q+\varphi_{\geq3}(t,x,y,\eta_*+q).$$
    The imaginary part of $\varphi_{\eta\eta}(t,x,y,\eta_*)$ is positive definite when $\mathcal B>0$ and $\sin(\omega t) \neq 0$; indeed,
    $$\text{Im} \varphi_{\eta\eta}(t,x,y,\eta_*)=\mathcal B\left(\frac{\sin(\omega t)}{\omega}\right)^2I.$$
    Therefore the local complex Gaussian integral is convergent and has the factor
    \begin{align}\label{factor}(2\pi\alpha)\det\left(\frac{\varphi_{\eta\eta}(t,x,y,\eta_*)}{i}\right)^{-1/2}=(2\pi\alpha)\left(\frac{i\sin(\omega t)}{\omega}\right)^{-1}(\det Z(t,y,\eta_*))^{-1/2}.
    \end{align}
    The Gaussian Taylor identity implies
 \begin{multline}\label{GT}
        \int_{\mathbb{R}^2}
e^{\frac{i}{2\alpha}q\cdot\varphi_{\eta\eta}(t,x,y,\eta_*)q}F(q)dq\\
\sim (2\pi\alpha)\det\left(\frac{\varphi_{\eta\eta}(t,x,y,\eta_*)}{i}\right)^{-1/2}\sum_{\nu=0}^\infty \frac{i^{-\nu}\alpha^\nu}{2^\nu \nu!}\left\langle \varphi^{-1}_{\eta\eta}D_\eta,D_\eta\right\rangle^{\nu}F(0).
\end{multline}
    Now expand the non-quadratic part of the phase
    $$
    e^{i\varphi_{\geq3}/\alpha}=\sum_{\mu\geq 0}\frac{i^\mu}{\mu!}\alpha^{-\mu}\varphi^\mu_{\geq 3}.
    $$
Applying \eqref{GT} to 
    $$
    F(q)=\varphi_{\geq 3}^\mu(t,x,y,\eta_*+q)a_{(k)}(t,x,y,\eta_*+q)
    $$
    produces the power $\alpha^{\nu-\mu+k}$, thus 
    $$\nu-\mu+k=m.$$
    Writing $j=m-k$ gives $\nu=j+\mu$. Since 
    $$\varphi_{\geq 3}^\mu=O(|\eta-\eta_*|^{3\mu}),$$
    and $\left\langle \varphi^{-1}_{\eta\eta}D_\eta,D_\eta\right\rangle^{\nu}$ differentiates $2\nu$ times. Then the term survives if
    $$2\nu\geq 3\mu,$$
    equivalently,
    $$\mu\leq 2j,$$
    and the surviving terms are exactly \eqref{L}.
    Finally, plugging \eqref{factor} back into the identity we obtain
\begin{multline*}(2\pi \alpha)^{-2}(\det Z(t,y,\eta_*))^{1/2}(2\pi\alpha)\left(\frac{i\sin(\omega t)}{\omega}\right)^{-1}(\det Z(t,y,\eta_*))^{-1/2}\\
=\frac{\omega}{2\pi i\alpha\sin(\omega t)}.
\end{multline*}
    This proves \eqref{new}. The expression for \eqref{S} follows by substituting \eqref{critical_points} back to $S(t,y,\eta_*)$.
\end{proof}

\medskip
\section{More on the stationary phase method} \label{Sec:6}

\noindent
We are interested in asymptotic behaviour as $\alpha\to 0$
of the function
$$
 \int U_N (t,x,y) \rho(y) e^{iy\cdot\eta_0/\alpha} dy,
$$
where $\rho \in C_0^\infty(\Bbb R^2)$, $\rho\ge0$.

\medskip
\noindent
We have

\begin{theorem}
Let $U_N$ be the approximate solution from Theorem \ref{main1} and let a point $(t,y_0,\eta_0)$
be such that $y_0= y_0 (t,x,\eta_0)$ is the solution of the following system of equations
\begin{align*}
&x\wedge y= - (y\wedge \eta_0 - b)\, \frac{\sin(\omega t)}{\omega}, \\
&x\cdot y = \cos(\omega t) |y|^2 + \frac{\sin(\omega t)}{\omega} \eta_0 \cdot y
\end{align*}
defined by the equation $x=x^t(y,\eta_0)$. Furthermore, assume $\det x^t_y(y_0,\eta_0)\neq 0$.
Then for any function 
$$
\rho \in C_0^\infty \left(\Bbb R^2\setminus(\{0\}\cup \{y:y\wedge\eta_0 = b\})\right),
$$ 
we have
\begin{multline*}
\int  U_N(t,x,y) \rho (y) e^{iy\cdot \eta_0/\alpha} dy \\
= \rho(y_0) e^{i y_0\cdot \eta_0/\alpha} 
e^{iS (t, y_0 , \eta_0)/\alpha} \,
e^{-i{\pi/2} \,m(t,y_0)} \left|\cos^2(\omega t) -\frac{b^2}{\omega^2} |y_0|^{-4}\, \sin^2(\omega t)\right|^{-1/2} \\
+  O (\alpha), 
\end{multline*}
where
$m(t,y,\eta)\equiv m(\gamma^t)$ is the Morse index of the trajectory 
$$
\gamma^t = \{ x^s (y,\eta): \, \, 0\le s\le t\}.
$$
\end{theorem}

\begin{proof}
Let us consider the expression
$$
(2\pi \alpha)^{-2} \int_{\Bbb R^2} \int_{\Bbb R^2} e^{i\psi/\alpha} \rho(y) \, d\eta dy,
$$
where  $\psi = \varphi + y\cdot\eta_0$.

\noindent
We apply the stationary phase expansion to this integral. Let us first find stationary points of the phase $\psi$ in $\eta$ and $y$.

\noindent
Due to Proposition \ref{phi0} the solution of the equation $\psi_\eta = \varphi_\eta = 0$ 
is achieved on the set $x=x^t(y,\eta_0)$, that is set of points $(x,y,\eta_0)$ such that 
\begin{align}\label{st_point}
\begin{split}
&x\wedge y = - (y\wedge \eta_0 - b)\, \frac{\sin(\omega t)}{\omega},\\
&x\cdot y = \cos{\omega t} |y|^2 + \frac{\sin(\omega t)}{\omega} \eta\cdot y.
\end{split}
\end{align}
Furthermore using Proposition \ref{SProp} we find
\begin{multline}
\psi_y|_{x=x^t}  = \left(\varphi_y + \eta_0\right)|_{x=x^t} \\
= \left(S_y + \eta_0 - x_y^t \xi^t + (x-x^t) (\xi^t_y - i \mathcal B\, x_y^t)\right)|_{x=x^t}  \\
=  - \eta + \eta_0 = 0.
\end{multline}
Therefore the only stationary point is $z_0 = (y_0, \eta_0)$, where $y_0 = y_0(t,x,\eta_0) $
is the solution of the system of equations \eqref{st_point}  
where due to \eqref{y_eta} $y\wedge \eta_0 - b\not=0$. 

\noindent
We now compute the Hessian of the phase function $\psi$ at the stationary point $z_0$.
Note that
$$
{\rm Hess}\,\psi = {\rm Hess}\,\varphi
$$
and at the stationary point $z_0$ 
\begin{align*}
& x_y^t = \cos(\omega t) \Bbb I - A_y(y_0)\, \frac{\sin(\omega t)}{\omega} \\
&\xi^t_y = - \omega \sin(\omega t) \Bbb I -  \cos(\omega t) A_y(y_0) + A_x(x^t(y_0,\eta_0))\, x_y^t(y_0,\eta_0).
\end{align*} 
Denote by $Y$ the matrix-function 
$$
Y(t,y,\eta) = \xi^t_y(y,\eta) - i\, \mathcal B\, x^t_y(y,\eta).
$$ 
Then 
\begin{equation}\label{Hess}
{\rm Hess}\,\varphi = - 
\begin{pmatrix}
Y\, x^t_y & I+ Y\,x^t_\eta\\
Z\, x^t_y & Z\, x^t_\eta
\end{pmatrix}.
\end{equation}
Clearly $\varphi_{\eta y} = (\varphi_{y\eta})^T$ which means that 
$$
I+ Y\,x^t_\eta = (Z\, x^t_y)^T.
$$
Factorising the Hessian \eqref{Hess} 
$$
{\rm Hess}\,\varphi = -
\begin{pmatrix}
Y\, x^t_y & \Bbb I\\
Z\, x^t_y & 0
\end{pmatrix} 
\begin{pmatrix}
\Bbb I & (x^t_y)^{-1} x^t_\eta \\
0 & \Bbb I
\end{pmatrix} 
$$
we find
$$
{\rm det } ( {\rm Hess}\,\varphi ) = {\rm det } \, Z \cdot {\rm det }\, x^t_y = {\rm det } \, (\xi^t_\eta - i \mathcal B\, x^t_\eta) 
 \cdot {\rm det }\, x^t_y  
$$
Applying the stationary phase formula (see e.g. \cite{H2}, Section 7.7) we first note $({\rm det}\, Z)^{1/2}$ appearing in the asymptotic expansion in Theorem \ref{main1} cancels with  $({\rm det}\, Z)^{-1/2}$. Therefore the term $({\rm Hess}\, \varphi)^{-1/2} $ appearing in the stationary phase formula is reduced to 
the term, see \cite{LapSig},
$$
({\rm det}\, x^t_y)^{-1/2} = e^{-i{\pi/2} \,m(t,y_0)} \left|\cos^2(\omega t) -\frac{b^2}{\omega^2} |y_0|^{-4}\, \sin^2(\omega t)\right|^{-1/2}.
$$
\end{proof}

\end{document}